\documentclass[12pt,reqno]{amsart}

\usepackage[all]{xy}
\usepackage{latexsym}
\usepackage{rotating}
\usepackage{amsfonts}
\usepackage{amssymb}
\usepackage{amsmath}
\usepackage{hyperref}

\newtheorem{theorem}{Theorem}[section] 
\newtheorem{lemma}[theorem]{Lemma}
\newtheorem{corollary}[theorem]{Corollary}
\newtheorem{proposition}[theorem]{Proposition}
\newtheorem{question}[theorem]{Question}
\newtheorem{conjecture}[theorem]{Conjecture}

\theoremstyle{definition}
\newtheorem{definition}[theorem]{Definition}
\newtheorem{example}[theorem]{Example}
\newtheorem{remark}[theorem]{Remark}

\numberwithin{equation}{section}

\newcommand{\ol}{\overline}
\newcommand{\fX}{\mathfrak{X}}

\newcommand{\C}{\mathbb{C}}
\newcommand{\hm}{\mathrm{Hom}}
\newcommand{\SL}{\mathrm{SL}}

\newcommand{\GL}{\mathrm{GL}}
\newcommand{\PGL}{\mathrm{PGL}}
\newcommand{\PSL}{\mathrm{PSL}}

\renewcommand{\L}{\mathcal{L}}
\newcommand{\quot}{/\!\!/}

\newcommand{\R}{\mathbb{R}}
\newcommand{\Z}{\mathbb{Z}}

\newcommand{\bbZ}{\mathbb{Z}}

\newcommand{\bbC}{\mathbb{C}}

\def\co{\colon\thinspace}

\newcommand{\cross}{\times}

\newcommand{\srm}[1]{\stackrel{#1}{\maps}}
\newcommand{\maps}{\longrightarrow}
\newcommand{\tr}{\mathrm{tr}}
\newcommand{\e}{\emph}

\newcommand{\la}{\langle}
\newcommand{\ra}{\rangle}
\newcommand{\fd}{\mathfrak{d}}
\newcommand{\fg}{\mathfrak{g}}

\newcommand{\isom}{\cong}
\renewcommand{\H}{\mathcal{O}}
\newcommand{\fl}{\mathfrak{l}}

\newcommand{\beqa}{\begin{eqnarray*}}
\newcommand{\eeqa}{\end{eqnarray*}}

\newcommand{\cO}{{\mathcal O}}

\newcommand{\Hom}{\mbox{${\rm Hom}$}}

\begin{document}

\title[Compactification of Character Varieties]{Wonderful Compactification of Character Varieties}

\date{\today}

\author[I. Biswas]{Indranil Biswas}

\address{School of Mathematics, Tata Institute of Fundamental
Research, Homi Bhabha Road, Mumbai 400005, India}

\email{indranil@math.tifr.res.in}

\author[S. Lawton]{Sean Lawton}

\address{Department of Mathematical Sciences, George Mason University, 4400 
University Drive, Fairfax, Virginia 22030, USA}

\email{slawton3@gmu.edu}

\author[D. Ramras]{Daniel Ramras\\ \\
{\tiny Appendix by Arlo Caine and Sam Evens}}

\address{Department of Mathematical Sciences, Indiana University-Purdue University 
Indianapolis, 402 N. Blackford, LD 270, Indianapolis, IN 46202, USA}

\email{dramras@math.iupui.edu}

%\author[S. Evens]{Sam Evens}
\address{Department of Mathematics, University of Notre Dame, 255 Hurley, Notre Dame, IN, 46556, USA}
\email{sevens@nd.edu}

%\author[A. Caine]{Arlo Caine}
\address{Mathematics and Statistics, California State Polytechnic University, 3801 West Temple Avenue, Pomona, CA 91768, USA}
\email{jacaine@cpp.edu}


 
\begin{abstract}
Using the wonderful compactification of a semisimple adjoint affine algebraic group $G$ defined over an algebraically closed field $\Bbbk$ of arbitrary characteristic, we construct a natural compactification $\overline{\fX_\Gamma(G)}$ of the $G$-character variety of any finitely generated group $\Gamma$. When $\Gamma$ is a free group, we show that this compactification is always simply connected with respect to the \'etale fundamental group, and when $\Bbbk\,=\, \C$ it is also topologically simply connected. For other groups $\Gamma$, we describe conditions for the compactification of the moduli space to be simply connected and give examples when these conditions are satisfied, including closed surface groups and free abelian groups when $G\,=\,\PGL_n(\C)$. Additionally, when $\Gamma$ is a free group we identify the boundary divisors of $\overline{\fX_\Gamma(G)}$ in terms of previously studied moduli spaces, and we construct a family of Poisson structures on $\overline{\fX_\Gamma(G)}$ and its boundary divisors arising from Belavin--Drinfeld splittings of the double of the Lie algebra of $G$.  In the appendix, we explain how to put a Poisson structure on a quotient of a Poisson algebraic variety by the action of a reductive Poisson algebraic group.
\end{abstract}

\maketitle

\section{Introduction}

To understand how groups $\Gamma$ act on spaces $X$ one considers homomorphisms $\Gamma\longrightarrow \mathrm{Aut}(X)$. When $\mathrm{Aut}(X)$ is an algebraic group $G$, the 
collection of homomorphisms $\hm(\Gamma, G)$ is an algebraic variety and so deformation 
techniques are available. From the associated study of $G$-local systems, two homomorphisms are 
equivalent when they are conjugate via an element of $G$. In this case, the quotient space 
$\hm(\Gamma, G)/G$ is naturally considered. Unfortunately this quotient space is not generally algebraic 
and so deformation techniques are not available. An approximation to this space, that often has 
better properties, is called the $G$-character variety of $\Gamma$. It will be denoted by 
$\fX_\Gamma(G)$.

When $G$ is a reductive algebraic group over an algebraically closed field $\Bbbk$, the above
mentioned space
$\fX_\Gamma(G)$ is precisely the geometric invariant
theoretic (GIT) quotient $\hm(\Gamma, G)\quot G$; in other words, it is the spectrum of the 
ring of invariants $\Bbbk[\hm(\Gamma,G)]^G$.

Considering families lying in $\fX_\Gamma(G)$ demands an understanding of (geometrically 
meaningful) boundary divisors, and as such compactifications of $\fX_\Gamma(G)$ arise naturally.

For example, in \cite{MorSha}, a compactification of $\SL_2(\C)$-character varieties by actions 
on $\R$-trees gave a new proof of Thurston's theorem that projective measured geodesic laminations give a compactification of Teichm\"uller space; the latter gives a classification of surface group 
automorphisms. Extensions of these ideas to real Lie groups were considered by Parreau~\cite{Parreau}.
More recently, in \cite{compactmanon}, it was shown that each quiver-theoretic 
avatar of a free group character variety developed in \cite{FlLa3} determines a natural 
compactification, under the assumption that $G$ is simple and simply connected over $\C$.
And in \cite{Kom}, compactifications of relative character varieties of punctured spheres are considered in order to understand the relationship between the Dolbeault moduli space of Higgs bundles and the Betti moduli space of representations.

In this paper, we prove the following theorem.

\begin{theorem}\label{intro-thm}
Let $G$ be a semisimple algebraic group of adjoint type defined over an algebraically closed field
$\Bbbk$. Then the wonderful compactification of $G$ determines a compactification of $\fX_\Gamma(G)$ for any finitely generated group $\Gamma$. If $\Gamma$ is a free group, then this compactification is \'etale simply connected. Moreover, when $\Bbbk=\C$ there exists a compactification of $\fX_\Gamma(G)$ that is both topologically and \'etale simply connected whenever $\fX_\Gamma(G)$ is simply connected and normal.
\end{theorem}

This result follows from Theorem~\ref{etaletheorem}, Corollary~\ref{sc-cor} and Lemma~\ref{normal-lem}.  Let $\overline{\fX_\Gamma (G)}$ denote the compactification of $\fX_\Gamma(G)$ from Theorem \ref{intro-thm}.  In Proposition~\ref{sc-wonderful-cor}, we apply Theorem~\ref{intro-thm} to prove the following corollary.

\begin{corollary} Let $G$ be a semisimple algebraic group of adjoint type over $\C$. Then $\overline{\fX_\Gamma (G)}$ is both topologically and \'etale simply connected if:
\begin{enumerate}
\item $\Gamma$ is a free group,
\item $\Gamma$ is a surface group and $G\,=\,\PGL_n(\C)$, or
\item $\Gamma$ is free abelian and $G$ does not have exceptional factors.
\end{enumerate}
\end{corollary}

In Sections \ref{divisor} and \ref{poisson} we further study the case in which $\Gamma$ is a free group. We identify the boundary divisors of $\overline{\fX_\Gamma(G)}$ (Theorem~\ref{div-thm}) in terms of the \e{parabolic character varieties} studied by Biswas--Florentino--Lawton--Logares~\cite{BFLL}, and we construct a Poisson structure on $\overline{\fX_\Gamma(G)}$ and on its boundary divisors (Theorem~\ref{poisson-thm}) using work of Evens--Lu~\cite{EL1, EL2}, who constructed a Poisson structure on $\ol{G}$. To show there is a Poisson structure on $\overline{\fX_\Gamma(G)}$, we utilize recent work of Lu--Mouquin~\cite{Lu-Mouquin} to equip $\ol{G}^{\,r}$ with a Poisson structure for which the diagonal conjugation action of $G$ is a Poisson action (for an appropriate Poisson Lie group structure on $G$). To show that this Poisson structure descends to $\overline{\fX_{\Gamma}(G)}$, we use the fact that when a reductive algebraic Poisson group acts on a projective Poisson variety and the action is Poisson, then the GIT quotient inherits a Poisson structure.  This fact, although known to experts, does not appear in the literature.  The appendix, written by Arlo Caine and Sam Evens, provides a proof of this fact.

\section*{Acknowledgements}

We thank S. Evens, J.-H. Lu, C. Manon, and A. Weinstein for helpful discussions, an anonymous referee for identifying an error in a previous version of this paper, and two subsequent referees for helping improve exposition and clarity.  The use of mixed product structures in Section 6 was suggested by Evens.   We thank him for explaining these ideas.  We also thank the International Centre for Theoretical Sciences in Bangalore, India for  hosting us during the workshop on Higgs Bundles from 21 March 2016 to 01 April 2016, where some of this work was conducted. Biswas is supported by the J. C. Bose Fellowship, Lawton is supported by the Simons Foundation, USA (\#245642), and Ramras is supported by the Simons Foundation, USA (\#279007). Lastly, Lawton and Ramras acknowledge support from NSF grants DMS 1107452, 1107263, 1107367 ``RNMS: GEometric structures And Representation  varieties" (the GEAR Network).

\section{Wonderful Compactification of Groups}

Let $G$ be a connected affine algebraic group defined over an algebraically closed field $\Bbbk$;
there is no condition on its characteristic. Let $\mathfrak{g}\,=\,\mathrm{Der}_\Bbbk(\Bbbk[G],\Bbbk)^G$ be the Lie algebra of $G$, where $G$ acts on the derivations via the left--translation action of $G$ on itself. The group $G$ is said to be of {\it adjoint type} if the adjoint representation
\begin{equation}\label{e1} \rho\, :\, G\, \longrightarrow\,\text{GL}(\mathfrak{g})
\end{equation} is an embedding. The center of a  group of adjoint type is trivial.

We will always assume that $G$ is semisimple of adjoint type. Therefore, $G$ is of the
form $\prod_{i=1}^m (G_i/Z_i)$, where each $G_i$ is a simple simply connected group and
$Z_i$ is the center of $G_i$.

A {\it compactification} of a variety $X$ is a complete variety $Y$ with $X$ as 
a dense open subset.  In \cite{DP}, assuming the base field is of characteristic 0, a compactification of $G$ is 
constructed, called the {\it wonderful compactification}. In \cite{strict} the construction is 
generalized to arbitrary characteristic. Denote the wonderful compactification of $G$ by $\overline{G}$.
In \cite{EL1, EL2}, a Poisson structure on $\overline{G}$ is constructed when
the characteristic of the base field is zero.

We now describe the construction of $\overline{G}$, following the exposition in \cite{EL1, EvJo}.
Let $n$ be the dimension of $G$.
The general linear group $\GL(\mathfrak{g} \oplus \mathfrak{g})$ acts on the space of $n$-dimensional 
subspaces of $\mathfrak{g} \oplus \mathfrak{g}$ transitively with the stabilizer of a
point being a parabolic subgroup $P$. The Grassmannian $\mathrm{Gr}(n, \mathfrak{g} \oplus 
\mathfrak{g})\,=\,\GL(\mathfrak{g} \oplus \mathfrak{g})/P$ of dimension $(2n)^2-3n^2\,=\,n^2$
parametrizes the $n$-dimensional subspaces of $\mathfrak{g} \oplus \mathfrak{g}$.
Consider the composition homomorphism
$$
G\times G \, \stackrel{\rho\times\rho}{\longrightarrow} \GL(\mathfrak{g})\times
\GL(\mathfrak{g})\, \hookrightarrow\, \GL(\mathfrak{g} \oplus \mathfrak{g})\, ,
$$
where $\rho$ is the homomorphism in \eqref{e1} and $\GL(\mathfrak{g})\times
\GL(\mathfrak{g})$ is the subgroup of automorphisms of $\mathfrak{g} \oplus \mathfrak{g}$
that preserves its decomposition. This homomorphism produces an action of
$G \times G$ on $\mathrm{Gr}(n, \mathfrak{g} \oplus \mathfrak{g})$. Let
$$\mathfrak{g}_\Delta \,:=\, \{(x, \,x) \ \mid \ x \,\in\, 
\mathfrak{g}\}\, \subset\, \mathfrak{g} \oplus \mathfrak{g}$$ be the diagonal subalgebra, which
is an $n$-dimensional subspace and hence a 
point in $\mathrm{Gr}(n, \mathfrak{g} \oplus \mathfrak{g})$. The stabilizer of
$\mathfrak{g}_\Delta$ with respect to the above action of $G\times G$ on $\mathrm{Gr}(n, \mathfrak{g}
\oplus \mathfrak{g})$ is $$G_\Delta \,:=\, \{(g,\, g) \ \mid\ g \,\in\, G\}\, .$$ 
Therefore, the orbit of $\mathfrak{g}_\Delta$ is
$$(G\times G)\cdot\mathfrak{g}_\Delta\,=\, (G \times G)/G_\Delta \,\isom \, G\, .$$

The wonderful compactification of $G$ is then $\overline{G} \,=\, \overline{(G \times 
G)\cdot\mathfrak{g}_\Delta}$ , where the closure is taken inside $\mathrm{Gr}(n, \mathfrak{g} \oplus 
\mathfrak{g})$, making $\overline{G}$ an irreducible projective variety containing $G\,=\,
(G\times G)\cdot\mathfrak{g}_\Delta$ as a Zariski open subvariety.

\begin{theorem}[\cite{strict},~\cite{DP}]\label{wonderfulgroup}
The following properties hold for the wonderful compactification $\overline{G}$:
\begin{enumerate}
\item The action of $G\times G$ on $G$, defined by $(g_1, \,g_2)\cdot x\,=\, g_1xg^{-1}_2$, extends to a $G\times G$ action on $\overline{G}$ with $2^r$ orbits, where $r\,=\, {\rm rank}(G)$;

\item $\overline{G}$ is smooth, as is each $G\times G$ orbit closure in $\overline{G}$;

\item The complement $\overline{G}\setminus G$ consists of $r$ smooth divisors $D_1,\ldots ,D_r$ with simple normal crossings, each of which is the closure of a single $G\times G$ orbit.
\end{enumerate}
\end{theorem}

\begin{remark}\label{homogeneous}
In \cite{DP}, a canonical compactification (called the wonderful compactification) is constructed for certain homogeneous spaces $H/K$ called symmetric varieties, where $K$ is the fixed locus of an involution.  The wonderful compactification of $G$ above is a special case of this more general construction since the diagonal copy of $G$ inside $G\times G$, denoted $G_\Delta$, is the fixed locus of the involution $(a,b)\mapsto (b,a)$.  Then $G\cong (G\times G)/G_\Delta$, and the left action of $H$ on $H/K$ extends to the wonderful compactification of $H/K$ and becomes the $G\times G$ action on $\overline{G}$ after passing through this isomorphism.  We note this generalization since we will be referring to properties about this more general construction later in this paper.
\end{remark}

Note that the diagonal $G_\Delta\cong G$ acts by conjugation on $\overline{G}$.
We now show that $\overline{G}$ is simply connected, after reminding the reader of requisite terms.

A morphism of irreducible normal projective varieties $f\,:\,Y\,\rightarrow\, X$ is {\it 
\'etale} if the induced map $\widehat{\mathcal{O}_{f(y)}}\,\rightarrow\, 
\widehat{\mathcal{O}_y}$ between complete local rings is an isomorphism for all points $y\,\in\, 
Y$. An \'etale morphism $f$ is {\it Galois} if the induced injection on quotient 
fields $\Bbbk(X)\,\rightarrow\,\Bbbk(Y)$ is a Galois extension. The Galois
group for this extension acts on $Y$ with $X$ being the quotient. A {\it Galois covering} of 
$X$ is a finite Galois \'etale map $Y\,\rightarrow\, X$. We say $X$ is {\it \'etale simply 
connected} if it does not admit any non-trivial Galois coverings. Over $\C$, if the topological 
fundamental group of $X$ (in the strong topology) is trivial, then the \'etale fundamental group 
is trivial \cite{Milne}.
 
\begin{corollary}\label{cor-sc}
The variety $\overline{G}$ is \'etale simply connected. When $\Bbbk\,=\, \C$, the
topological fundamental group of $\overline{G}$ is trivial. 
\end{corollary}
 
\begin{proof}
Recall that $G$ is an open dense affine subvariety of $\overline{G}$. Since we are over 
an algebraically closed field, the Bruhat decomposition gives an affine cell in $G$ that 
is open and dense \cite{Borel}. So $\overline{G}$ is birational to affine space, which itself
is birational to projective space. Therefore, $\overline{G}$ is a rational variety. In 
general a projective, smooth, rational variety over an algebraically closed field is 
\'etale simply connected \cite{Kollar}. Thus, $\overline{G}$ is \'etale simply 
connected.

When $\Bbbk\,=\, \C$, the topological fundamental group of $\overline{G}$ is trivial, 
because $\overline{G}$ is a rational variety \cite[p.~483, Proposition 1]{Se}.
\end{proof}

\begin{remark}
Our proof of Corollary \ref{cor-sc} shows that any smooth compactification of $G$ is \'etale simply connected, and topologically simply connected over $\C$.
\end{remark}

\begin{example}\label{ex-group}
In the case of $G\,=\,\PSL_2(\C)\,=\,\PGL_2(\C)$, we have $\overline{G}\,=\,\mathbb{P}(M_2(\C))
\,=\,\C P^3$ where $M_2(\C)$ is the monoid of $2\times 2$ complex matrices. Naturally
$\PSL_2(\C)\,\subset \,\mathbb{P}(M_2(\C))$ and the action of $\PSL_2(\C)\times \PSL_2(\C)$
on $\PSL_2(\C)$ defined by $(g_1, \,g_2)\cdot x\,=\, g_1xg^{-1}_2$ extends to
an action on $\mathbb{P}(M_2(\C))$. The complement $D\,=\,\mathbb{P}(M_2(\C))\setminus \PGL_2(\C)$ is
the divisor $$\left(\{X\,\in \,M_2(\C)\ \mid\ \det(X)\,=\,0\}\setminus \{\mathbf{0}\}\right)/\C^*
\,=\,\left(\{(a,b,c,d)\,\in\,\C^4\ \mid \ ad\,=\,bc\}\setminus \{\mathbf{0}\}\right)/\C^*$$
which is the image of $\C P^1\times \C P^1$ under the Segre Embedding.
In this divisor, the locus of $a\neq 0$ is an affine open $\C^2$, and when $a\,=\,0$ we
have two copies of $\C P^1$ intersecting at the point $[(0,0,0,1)].$ 
\end{example}

\section{Wonderful Compactification of Character Varieties}\label{wonchar}

In this section, given a finitely generated group $\Gamma$ and a semisimple
algebraic group $G$ of adjoint type we construct a compactification of the $G$-character variety of $\Gamma$. There is no assumption on the characteristic of the algebraically closed base field $\Bbbk$.

First however we remind the reader of the basic terms and theorems of projective GIT.  A $G$-linearized line bundle over a $G$-variety $X$ is a line bundle $L$ over $X$ such that the projection map $L\to X$ is $G$-equivariant, and where the zero section of $L$ is $G$-invariant.  A point $x\in X$ is 
{\it semistable} with respect to $L$ if there exists a $G$-invariant invariant section $s:X\to L^{\otimes m}$ so 
$s(x)\not=0$ and the principal open $U_s$ defined by $s$ is affine.  If 
additionally the stabilizer at $x$ is finite and all $G$-orbits in $U_s$ are 
closed then $x$ is called {\it stable}.  Any point that is not semistable is called {\it unstable}. If there exists a basis $\{s_0,...,s_n\}$ for the space of sections of $L$ over $X$ such that the the map $x\mapsto (s_0(x),...,s_n(x))$ is a closed embedding into $\mathbb{P}^n$ then we say $L$ is {\it very ample}.  If $L^{\otimes m}$ is very ample for some positive $m$, then we say $L$ is {\it ample}.  An algebraic variety $X$ is isomorphic to a quasi-projective variety if and only if there exists an ample line bundle over $X$.  Given a $G$-linearized line bundle $L$ over $X$, there always exists a GIT quotient $X^{ss}_L\to X\quot_L G:=X^{ss}_L\quot G$, where $X^{ss}_L$ is the set of semistable points in $X$.  Moreover, $X\quot_L G$ is in general quasi-projective (see \cite[Theorem 1.10]{MFK} or \cite[Theorem 8.1]{Do}) and is projective if $X$ was projective and $L$ was ample to begin with (see \cite[Proposition 8.1]{Do}).

We begin constructing our compactifications with the case of a free group. Let $\Gamma\,=\,F_r$ be the free group of rank $r$ (we call the {\it standard} presentation of $F_r$ the one with no relations). With respect to the standard presentation, the evaluation map gives a bijection $\hm(F_r,\,G)\,\cong\, G^r$. Therefore, as the adjoint action of $G$ on $G$ extends to $\overline{G}$, the diagonal adjoint action of $G$ on $G^r$ also extends to the product $\overline{G}^{\,r}$.  Precisely, the action of $g\,\in\, G$ sends any $(x_1,\ldots ,x_r)\,\in\, \overline{G}^{\,r}$ to $(gx_1g^{-1},\ldots ,gx_rg^{-1})$. Thus, $\hm(F_r,\,G)$ is an affine Zariski open $G$-invariant subset of the $G$-variety $\overline{G}^{\,r}$; that is, $\overline{G}^{\,r}$ is a compactification of $\hm(F_r,\,G)$.

With respect to an ample line bundle $L$, the GIT quotient $\overline{G}^{\,r}\quot_L G$ is a projective variety.  We claim there is a line bundle that makes it a compactification of $\fX_{F_r}(G)$.

To establish this we prove a lemma that will also be relevant in Section \ref{divisor}, where we discuss divisors.

\begin{lemma}\label{linelemma} Let $G$ be a semisimple algebraic group of adjoint type, and let $\overline{G}$ be the wonderful compactification of $G$.  Then there is an ample line bundle $L$ on $\overline{G}$ so the divisors $\overline{G}\setminus G$ are the zero locus of a $G\times G$-invariant section of $L$.
\end{lemma}

\begin{proof}
We follow the discussion in Section 3 of \cite{DCS}, making some slight notational changes.

Let $H$ be a semisimple adjoint-type algebraic group over a field $\Bbbk$ of arbitrary characteristic (not equal to 2) and let $\tilde{H}$ be a simply-connected cover of $H$.  Let $\iota:\tilde{H}\to H$ be the corresponding central isogeny.  Let $\sigma$ be an involution of $H$ and let  $K=\iota^{-1}(H^\sigma)$ where $H^\sigma$ is the fixed locus of $\sigma$.  Define $X:=\tilde{H}/K$; a {\it symmetric variety}.  In \cite{DP,DS}, a compactification of $X$, denoted $\overline{X}$, is constructed called the {\it wonderful compactification}.  It is a compactification of $X$ that is a $\tilde{H}$-wonderful variety in the sense of Luna \cite{Lu4}.

As noted in Remark \ref{homogeneous}, we can think of $\overline{G}$ as an example of the wonderful compactification of a symmetric variety where $\tilde{H}=\tilde{G}\times \tilde{G}$, $\sigma$ is the involution $(a,b)\mapsto (b,a)$, $H^\sigma=G_\Delta$, and $K$ is the inverse image of $G_\Delta$ by the central isogeny $\iota: \tilde{G}\times \tilde{G}\to G\times G$.  Then $\tilde{H}/K=(\tilde{G}\times \tilde{G})/\iota^{-1}(G_\Delta)\cong(G\times G)/G_\Delta\cong G$.

Returning to the more general setting, let $S$ be a maximal torus in $\tilde{H}$ such that $\sigma(s)= s^{-1}$ for all $s\in S$.  Denote $\Lambda_A=\mathrm{Hom}(A, \Bbbk^*)$ for any abelian group $A$, and let $S_K=S/(S\cap K)$.   In \cite[Sections 2.2 and 3.1]{DCS}, the authors construct a basis for $\Lambda_{S_K}$ consisting of \emph{simple restricted roots} $\tilde{\Delta}=\{\tilde{\alpha}_1,..., \tilde{\alpha}_{\ell}\}$, where $\ell$ is the dimension of $S$.  Let $\Delta_{\overline{X}}$ be the irreducible components of codimension 1 in $\overline{X}\setminus X$ (i.e., the divisors).  It is shown \cite[Theorem 3.2]{DCS} that there is a bijection between $\Delta_{\overline{X}}$ and $\tilde{\Delta}$ given by $D\mapsto j(\mathcal{O}(D))$ where $j:\mathrm{Pic}(\overline{X})\to \Lambda_{S_K}$ is a monomorphism and $\mathcal{O}(D)$ is the line bundle over $\overline{X}$ with section whose zero locus is $D$.  This correspondence extends to a bijection between subsets $\Gamma\subset \tilde{\Delta}$ and the set of $\tilde{H}$-orbit closures defined by $X_\Gamma:=\cap_{\{D|j(\mathcal{O}(D))\in \Gamma\}}D$.

From this, for each $\tilde{\alpha}\in \tilde{\Delta}$ there is a line bundle $\mathcal{L}_{\tilde{\alpha}}$ over $\overline{X}$ and an $\tilde{H}$-invariant section $s_{\tilde{\alpha}}$ of $\mathcal{L}_{\tilde{\alpha}}$ whose divisor is $X_{\tilde{\alpha}}$.  In our setting, $\overline{G}$ and each of its divisors are embedded in a Grassmannian, and so we may take $\mathcal{L}_{\tilde{\alpha}}$ to be ample.  Therefore, $\mathcal{L}_{\tilde{\alpha}_1}\otimes\cdots \otimes \mathcal{L}_{\tilde{\alpha}_{\ell}}$ is an ample line bundle over $\overline{X}$ whose section $s_{\tilde{\alpha}_1}\otimes\cdots \otimes s_{\tilde{\alpha}_{\ell}}$ is $\tilde{H}$-invariant and whose non-zero locus is exactly $X$.

Therefore, the same holds for the special case when $\overline{X}=\overline{G}$.  We note that the $\tilde{G}\times \tilde{G}$-action on $\overline{G}$ factors through the $G\times G$-action we consider given the isomorphism $(\tilde{G}\times \tilde{G})/\iota^{-1}(G_\Delta)\cong(G\times G)/G_\Delta$.

\end{proof}

\begin{theorem}
There exists an ample line bundle $\mathcal{L}$ on $\overline{G}^r$ so that $\overline{G}^{\,r}\quot_{\mathcal{L}} G$ is a compactification of $\fX_{F_r}(G)$.
\end{theorem}

\begin{proof}
Let $L$ be the line bundle on $\overline{G}$ and $s$ the invariant section from Lemma \ref{linelemma}.  Then $\mathcal{L}:=L^{\boxtimes r}$ is an ample line bundle on $\overline{G}^r$ with a $G\times G$-invariant section $s^{\boxtimes r}$ whose non-vanishing locus is $G^r$.  Therefore the GIT quotient $\overline{G}^r\quot_{\mathcal{L}} G$, which is a projective variety, is a compactification of $\fX_{F_r}(G)$.
\end{proof}

\begin{remark}
As in \cite{HeStarr}, which concerned the case of $r=1$, we suspect the above construction is independent of $\mathcal{L}$. Regardless, we will always use the line bundle $\mathcal{L}$ in our constructions, even if the notation is suppressed. 
\end{remark}

Now let $\Gamma$ be a finitely generated group, say with $r$ generators. Fixing $r$ generators,
there is a surjection $\varphi\,:\,F_r\,\longrightarrow\, \Gamma$ that induces an inclusion
$\varphi_\#\,:\, \fX_{\Gamma}(G)\,\hookrightarrow\, \fX_{F_r}(G)$.

\begin{definition}
The {\it wonderful compactification of} $\fX_\Gamma(G)$ is the closure of $\fX_\Gamma(G)$ in $\overline{G}^{\,r}\quot_{\mathcal{L}} G$ with respect to the above inclusion $\varphi_\#$. This compactification will be denoted by $\overline{\fX_\Gamma(G)}$. 
\end{definition}

Up to isomorphism $\fX_\Gamma(G)$ does not depend on $\varphi_\#$, however the compactification $\overline{\fX_\Gamma(G)}$ does depend on the choice of $\varphi$ (see \cite{Martin} for example).  In other words, since a presentation of $\Gamma$ is equivalent to $\varphi$, the compactification depends on a choice of a presentation for $\Gamma$.  

It would be interesting to explore how different presentations of $\Gamma$ change the geometry of the resulting divisors (the Zariski open subvariety $\fX_\Gamma(G)$ does not change up to isomorphism).  

With that said, it is perhaps surprising that some of our theorems concerning $\overline{\fX_\Gamma(G)}$ do not depend on the presentation of $\Gamma$.   Because of this, we will not always specify the presentation of $\Gamma$ in the statement of our theorems.

\begin{theorem}\label{etaletheorem}
With respect to the standard presentation of $F_r$, the wonderful compactification
$\overline{\fX_{F_r}(G)}$ is normal and \'etale simply connected. When $\Bbbk\,=\,\C$ it is
topologically simply connected.
\end{theorem}

\begin{proof}
Since the GIT quotient of a smooth variety is normal, and $\overline{G}^{\,r}$ is smooth, it follows
that $\overline{\fX_{F_r}(G)}\cong \overline{G}^{\,r}\quot_{\mathcal{L}} G$ is normal.

The quotient map $\overline{G}^{\,r}\,\longrightarrow\, \overline{G}^{\,r}\quot_{\mathcal{L}} G$ induces an
isomorphism of \'etale fundamental groups (and topological fundamental groups when $
\Bbbk\,=\,\C$) by \cite[Theorem 1]{biswasfundamental}. From Corollary \ref{cor-sc}
we know that $\overline{G}$ is \'etale simply connected and therefore the product $\overline{G}^{\,r}$ is
also \'etale simply connected. Consequently, $\overline{\fX_{F_r}(G)}\cong \overline{G}^{\,r}\quot_{\mathcal{L}} G$ is
\'etale simply connected.

If $\Bbbk\,=\,\C$, then $\overline{G}^{\,r}$ is topologically simply connected by Corollary 
\ref{cor-sc}. Hence $\overline{\fX_{F_r}(G)}$ is
topologically simply connected when $\Bbbk\,=\,\C$.
\end{proof}

\begin{example}\label{rank1}
By \cite[Theorem 0.7]{HeStarr}, in arbitrary characteristic $\overline{\fX_{F_1}(G)}\cong \overline{T}\quot W$ where $\overline{T}$ is the closure of a maximal torus $T\subset G$ in $\overline{G}$, $W\subset G$ is the Weyl group, and the quotient is independent of line bundle.
\end{example}

\begin{example}\label{ex-charvar}
Let $\mathcal{K}:=\Z/2\Z\times \Z/2\Z$ be the Klein 4-group. Consider $\fX_{F_2}(\PSL_2(\C))\cong \fX_{F_2}(\SL_2(\C))\quot\mathcal{K}$. By \cite{Sikora-notsimply}, $$\fX_{F_2}(\PSL_2(\C))\cong\C^3\quot\mathcal{K}\cong \mathrm{Spec}\left(\C[g_1, g_2, g_3, g_4]/(g_1g_2g_3 -g_4^2)\right),$$ where 
$$\fX_{F_2}(\SL_2(\C))\cong \{(\tr(A),\tr(B),\tr(AB))\ \mid\ 
A,B\in \SL_2(\C)\}\cong \C^3,$$ and $g_1$ corresponds to $\tr(A)^2$, $g_2$ to $\tr(B)^2$, $g_3$ to $\tr(AB)^2$, and $g_4$ to $\tr(A)\tr(B)\tr(AB)$. 
Given Example \ref{ex-group},
$\overline{\fX_{F_2}(\PSL_2(\C))}\cong(\C P^3\times \C P^3)\quot_{\mathcal{L}} \PSL_2(\C).$
\end{example}

\begin{remark}
In \cite[Theorem 3.4]{FlLa3} it is shown that to each connected quiver $Q$ and connected 
reductive complex algebraic group $G$, there is an algebraic variety $\mathcal{M}_Q(G)$ 
isomorphic to $\fX_{F_r}(G)$, where $r$ is the first Betti number of $Q$. In 
\cite[Theorem 1.1]{compactmanon} it is shown, in the case where $G$ is simple and simply 
connected, that each such $\mathcal{M}_Q(G)$ determines a generally distinct 
compactification of $\fX_{F_r}(G)$. When $Q$ has exactly one vertex the compactification 
in \cite{compactmanon} reduces to the GIT quotient of a product of compactifications of 
$G$, similar to the construction considered here for $\Gamma\,=\,F_r$. Now the 
compactification of the group $G$ considered in \cite{compactmanon} comes from its 
so-called Rees algebra. As shown in \cite[Example 8.1]{kavehmanon}, this 
compactification of $G$ coincides with the wonderful compactification of $G$. Therefore, 
our construction is a special case of the construction in \cite{compactmanon} in the 
overlapping situation when $\Gamma$ is free, and $G$ is a simple, simply connected, 
complex algebraic group of adjoint type $($exactly if $G$ is one of $G_2$, $F_4$, or 
$E_8$; see \cite{mo-simplyadjoint} for example$)$.
\end{remark}

\begin{remark} \label{commutebar}
In~\cite[Remark 4.6]{Kannan} it is shown that there is a natural isomorphism $\ol{G}^{\,r} \to \ol{G^r}$. 
For any semisimple algebraic group $H$ of adjoint type over an algebraically closed field,
Lusztig~\cite{Lusztig1, Lusztig2} introduced a partition of $\ol{H}$ into finitely many $H$--stable pieces (where $H$ acts by conjugation). Applied to the group $H=G^r \isom \Hom(F_r, G)$, this gives a partition of $\ol{G^r} \isom \ol{G}^{\,r}$ into $G^r$--stable pieces, which are automatically stable under the diagonal conjugation action of the diagonal subgroup $G\cong G_\Delta \subset G^r$. The closures of these $G^r$--stable pieces were investigated by He~\cite{He}.  It would be interesting to understand the images of these sets in $\ol{\fX_{F_r} (G)}$. 
\end{remark}

\section{Simply connected compactifications over \texorpdfstring{$\C$}{C}}

In this section we work over $\C$, and argue that in some cases we can normalize the wonderful compactification of $\fX_\Gamma(G)$ and obtain simply connected compactifications of character varieties when $\Gamma$ is not free.

We need the following standard result; see~\cite{ADH} and the references therein.

\begin{proposition}\label{sc-prop} 
If $Z$ is a normal projective variety, and  $A\,\subsetneq\, Z$ is a closed subvariety, then the natural homomorphism $\pi_1 (Z\setminus A) \,\longrightarrow\, \pi_1 (Z)$ is surjective.
\end{proposition}

\begin{corollary} \label{sc-cor}
Let $G$ be a semisimple algebraic group of adjoint type over $\C$, and let $\Gamma$ be either a
finitely generated free or free abelian group of rank $r$, or the fundamental group of a closed,
orientable surface. If $\overline{\fX_\Gamma (G)}$ is a normal compactification of $\fX_\Gamma (G)$,
then $\overline{\fX_\Gamma (G)}$ is simply connected. Consequently, $\overline{\fX_\Gamma (G)}$
is also \'etale simply connected.
\end{corollary}

\begin{proof}
For the allowed $G$ and $\Gamma$, it is shown in~\cite{BiLa, BLR} that $\pi_1 (\fX_\Gamma (G)) \,=\,
1$. The result now follows from Proposition~\ref{sc-prop}.
\end{proof}

The following two lemmas are standard.

\begin{lemma}\label{normal-lem}
If $A \subset Z$ is a nonempty Zariski open normal subset of an irreducible projective variety
$Z$, then the normalization $\widetilde{Z}$ of $Z$ contains an open subset
isomorphic to $A$. In particular, $\widetilde{Z}$ is still a
compactification of $A$.
\end{lemma}

\begin{lemma}\label{normalproduct}
Let $X$ and $Y$ be normal varieties over an algebraically closed field $\Bbbk$.
Then $X\times Y$ is also normal.
\end{lemma}

With the above lemmas and corollary in mind, we define the normalized wonderful compactification 
of a normal character variety $\fX_\Gamma(G)$ to be the normalization of 
$\overline{\fX_\Gamma(G)}$.

\begin{proposition} \label{sc-wonderful-cor}
Let $\fX^0_\Gamma (G)$ denote the component of $\fX_\Gamma (G)$ that contains the trivial
representation. In the following cases, the normalized wonderful compactification of
$\fX^0_\Gamma (G)$ is a simply connected compactification of $\fX^0_\Gamma (G)$ independent of the presentation of $\Gamma:$ 
\begin{enumerate}
\item $\Gamma \,=\, \bbZ^r$ and $G$ is any semisimple algebraic adjoint group with no exceptional
factors $;$

\item $\Gamma \,=\, \pi_1(\Sigma)$, with $\Sigma$ a closed orientable surface, and
$G \,=\, \mathrm{PGL}_n$.
\end{enumerate}
\end{proposition}

\begin{proof} 
We will show that in both these cases, the character variety $\fX_\Gamma (G)$ is normal. The 
result will then follow from Corollary~\ref{sc-cor} and Lemma~\ref{normal-lem}.

When $G \,=\, \mathrm{SL}_n, \mathrm{GL}_n$, $\mathrm{SO}_n$, or $\mathrm{Sp}_{2n}$, 
Sikora has shown that $\fX^0_{\bbZ^r} (G)$ is normal \cite[Theorem 2.1]{Si6}. Now since the 
left action of the center of $G$, denoted $Z(G)$, commutes with the conjugation action of $G$ on 
$\hm(\bbZ^r, G)$, we conclude $\fX_{\bbZ^r}(G/Z(G))\,\cong\, \fX_{\bbZ^r}(G)/Z(G)^r$. In view
of this, since 
normality is preserved under GIT quotients $\fX_{\bbZ^r} (G)$ is likewise normal for $G \,=\, 
\mathrm{PSL}_n\cong\mathrm{PGL}_n$, $\mathrm{PSO}_n$, or $\mathrm{PSp}_{2n}$.

Now let $G$ be a semisimple algebraic adjoint group with no exceptional factors. Then $G\,\cong \,
G_1\times \cdots \times G_n$, where each $G_i$ is isomorphic to a simple algebraic adjoint group
of type $A_n,\,B_n,\,C_n,\,D_n$. By Lemma \ref{normalproduct} and the previous paragraph
$\fX_{\bbZ^r} (G_1\times \cdots \times G_n)\,\cong\,\fX_{\bbZ^r} (G_1)\times \cdots
\times \fX_{\bbZ^r} (G_n)$ is normal.

In the second case, it is a result of Simpson that $\mathrm{Hom}(\pi_1(\Sigma), \GL_n)$ is a normal variety (see \cite{Simpson1, Simpson2}) . The group $\mathcal{Z} \,=\, 
\mathrm{Hom}(\pi_1(\Sigma), \,Z(\GL_n))$, which is isomorphic to
${\mathbb G}^{b_1(\Sigma)}_m$, acts on $\mathrm{Hom}(\pi_1(\Sigma), \,
\GL_n)$ by left multiplication, and we have $$\mathrm{Hom}(\pi_1 \Sigma,\, \GL_n)\quot \mathcal{Z} 
\,\cong\, \mathrm{Hom}^0(\pi_1(\Sigma),\, \mathrm{PGL}_n)\, ,$$ where the right-hand side denotes the 
identity component. Since the GIT quotient of a normal variety is normal, we find $\mathrm{Hom}^0(\pi_1(\Sigma),\, \mathrm{PGL}_n)$, and consequently $\fX^0_{\pi_1(\Sigma)} (\mathrm{PGL}_n)$, are normal.
\end{proof}

In \cite{BLR} we conjecture that for certain groups $\Gamma$ whose abelianization is free abelian 
(which we call {\it exponent canceling groups}), that $\fX^0_\Gamma(G)$ is simply connected (see 
\cite[Conjecture 2.7]{BLR}). We also expect that $\fX^0_\Gamma(G)$ is normal in these cases. 
Consequently, we now make:

\begin{conjecture}
The normalized wonderful compactification of $\fX^0_\Gamma (G)$ is a simply connected 
compactification of $\fX^0_\Gamma (G)$ for all exponent canceling $\Gamma$ and any 
semisimple adjoint type complex algebraic group $G$.
\end{conjecture}

\section{Boundary Divisors}\label{divisor}

In this section we continue to work over $\C$. Given a complex projective variety $X$ with a distinguished dense open affine subvariety $A \subset X$, we will use the term \e{boundary divisor} to refer to hypersurfaces of $X$ (that is, irreducible codimension 1 subvarieties) contained in $X\setminus A$. By Theorem \ref{wonderfulgroup}, the complement $\overline{G}\setminus G$ is a union of $r\,=\,\mathrm{rank}(G)$ smooth boundary divisors, and each of these divisors is the closure of a $G\times G$-orbit.

Now let $D_i$ be a boundary divisor of $\overline{G}$. Then there exist $$\mathfrak{m}_{I_1},\ldots, \mathfrak{m}_{I_{m_i}}\, \in\, \mathrm{Gr}(n, \mathfrak{g} \times \mathfrak{g})\, ,$$ where each $I_j\subset\{1,\ldots , r\}$, so that $$D_i\,=\,\cup_j(G\times G)\cdot \mathfrak{m}_{I_j}\,\cong\,\cup_j(G\times G)/\mathrm{Stab}(\mathfrak{m}_{I_j})\, .$$ In particular, each boundary divisor is isomorphic to a union of homogeneous spaces, each a quotient by a closed subgroup (since stabilizers of algebraic group actions are always algebraic subgroups).

Given a surjective, continuous map $q\co X\to Y$, we say that $A \subset X$ is saturated with respect to $q$ if $A = q^{-1} (q(A))$.
 
\begin{lemma}\label{gitdivisor}
Let $V$ be an affine $G$-variety and $W$ a compactification of $V$ on which the $G$-action extends. Let $L$ be an ample line bundle with a $G$-invariant section whose non-zero locus is exactly $V$. Assume that each boundary divisor of $W$ is saturated with respect to the GIT quotient map $W\to W\quot_L G$.  Then the boundary divisors of $W\quot_L G$, with respect to the open subvariety $V\quot G$, are exactly the components of $\left(W\setminus V\right)\quot_L G$.
\end{lemma}

\begin{proof}
As the $G$-action extends to $W$, we see that $V$ is a $G$-stable affine open subset of $W$, and the boundary divisors in $W\setminus V$ are unions of $G$-orbits. The usual gluing construction for the GIT quotient (see \cite[Section 8.2]{Do}) shows that $V\quot G$ is an affine open subvariety in $W\quot_L G$. Since the boundary divisors in $W\setminus V$ are saturated, $W\setminus V$ is itself saturated, so we find that $(W\quot_L G) \setminus (V\quot G)$ is exactly $\left(\cup_i D_i\right)\quot_L G$ where the $D_i$'s are the boundary divisors  in $W\setminus V$.
\end{proof}

In \cite{BFLL} parabolic character varieties of free groups are defined and studied. We recall their definition. Let $G$ be a complex reductive group, and let $G_1,\ldots ,G_m$ be closed subgroups. Then $G$ acts on the product
$$G^n\times \prod_{1\leq j\leq m}G/G_j$$
by
$$g\cdot(h_1,\ldots, h_n, g_1G_1,\ldots,g_mG_m)\,=\,
(gh_1g^{-1},\ldots,gh_ng^{-1}, gg_1G_1,\ldots,gg_mG_m)\, .$$ The quotient
$(G^n \times\prod_{1\leq j\leq m}G/G_j)\quot G$ is the parabolic character variety of the
free group of rank $n$ with parabolic data $\{G/G_j\}_{j=1}^m$. We note that when the $G_i$'s
are reductive, as assumed in \cite{BFLL}, the homogeneous spaces $G/G_i$ are affine, and when
the $G_i$'s are parabolic, the homogeneous spaces $G/G_i$ are projective. In general, the homogeneous spaces $G/G_i$
are quasi-projective \cite[Theorem 6.8]{Borel}.

\begin{theorem}\label{div-thm}
The boundary divisors in $\overline{G}^{\,r}\quot_{\mathcal{L}} G$ are unions of parabolic character varieties of free groups.
\end{theorem}

\begin{proof}
As noted above the boundary divisors in $\overline{G}$ are unions of homogeneous spaces of $G\cross G$, and by Theorem~\ref{wonderfulgroup} each boundary divisor is the closure of a single $G\cross G$--orbit. Therefore, $\overline{G}^{\,r}\setminus G^r$ consists of unions of products of $G\cross G$--homogeneous spaces. Since the conjugation action is a restriction of the $G\cross G$--action and by Lemma \ref{linelemma} there exists a $G\cross G$--equivariant section $s$ to $\mathcal{L}$ such that $G^r$ is the non-vanishing locus of $s$, the boundary divisors of $\overline{G}^r$ are saturated with respect to the GIT quotient map for the conjugation action.  The action of conjugation on an orbit corresponds, under the isomorphism between the orbit and the corresponding homogeneous space, to the left action on the homogeneous space. Thus, by Lemma \ref{gitdivisor} and the definition of parabolic character variety of free groups, the result follows.
\end{proof}

\begin{remark}
As shown in \cite{Ep}, the closure of an orbit in $\overline{G}$ under the conjugation action need not be a finite union of suborbits. Therefore, the boundary divisors in the previous theorem need not be finite unions of parabolic character varieties.
\end{remark}

\begin{example}\label{ex-div}
In Example \ref{ex-group} we see that the sole boundary divisor of the wonderful compactification of $\PSL_2(\C)$ is isomorphic to $\C P^1\times \C P^1$, a product of homogeneous spaces. Therefore, in Example \ref{ex-charvar}, given Theorem \ref{div-thm}, $\ol{\fX_{F_2}(\PSL_2(\C))}\setminus \fX_{F_2}(\PSL_2(\C))$ consists of GIT quotients of the diagonal left multiplication action of $\PSL_2(\C)$ on products of $\C P^1\times \C P^1$.  This is an example of a parabolic character variety as it is a left diagonal quotient of a product of homogeneous spaces.
\end{example}

It would be interesting to work out more examples (especially when $\Gamma$ is not free), or the above examples in more detail.  We leave this to future work.

\section{Poisson Structures}\label{poisson}

Recall that a Poisson algebra is a Lie algebra in which the Lie bracket  is also a derivation in each variable. We call a quasi-projective variety $X$ over $\C$ a Poisson variety if the 
sheaf of regular functions on $X$, denoted $\mathcal{O}(X)$, is equipped with the structure of a sheaf of Poisson algebras.
In this case, the sheaf of \e{holomorphic} functions on $X^{sm}$ (where $X^{sm}$ is the smooth locus of $X$) becomes a sheaf of Poisson algebras as well, making  $X^{sm}$ a complex Poisson manifold.  

The Poisson bracket on the algebra of holomorphic functions $\H (X^{sm})$ is induced by an exterior bivector field $\Lambda\,\in\, \Lambda^{2}(T^{1,0} X^{sm})$, see for instance \cite{Polishchuk}.
In other words, if $f,\,g\,\in\, \H(X^{sm})$, then the bracket is given by $\{f,g\}=\Lambda(df, dg).$ In local (complex) coordinates $(z_1,\ldots ,z_k)$ the bivector takes the form 
$$\Lambda=\sum_{i,j}\Lambda_{i,j}\frac{\partial}{\partial z_i}\land 
\frac{\partial}{\partial z_j}$$ and so 
\begin{eqnarray}\label{bivectorformula}
\{f,g\}&=\sum_{i,j}\left(\Lambda_{i,j}\frac{\partial}{\partial z_i}\land \frac{\partial}{\partial z_j}\right)\cdot\left(\frac{\partial f}{\partial z_i} dz_i \otimes \frac{\partial g}{\partial 
z_j}dz_j\right)\nonumber \\
&=\sum_{i,j}\Lambda_{i,j}\left(\frac{\partial f}{\partial z_i}\frac{\partial g}{\partial z_j}-
\frac{\partial f}{\partial z_j}\frac{\partial g}{\partial z_i}\right).
\end{eqnarray} 

In general, complex Poisson manifolds admit $(2,0)$--symplectic foliations \cite{LPV}.  For $f,\,g\,\in\, \H(X^{sm})$, the Hamiltonian vector field $H_f$ associated to $f$ is defined by $H_f(g)=\{f,g\}$.  Restricting the bivector $\Lambda$ to symplectic leaves gives the symplectic form $\omega(H_g,\,H_f)\,=\,\{f,\,g\}$.

For the rest of the section, $G$ will denote a semisimple algebraic group of adjoint type  over $\bbC$, with Lie algebra $\fg$. 
Let $\langle \langle\,,\, \rangle \rangle$ denote the Killing form on $\fg$. Following \cite{EL2}, we give the {\it double} $\fd := \fg \oplus \fg$ the symmetric, non-degenerate, and Ad-invariant bilinear form 
\begin{equation}\label{form} \la (x_1, x_2), (y_1, y_2) \rangle = \langle \langle x_1, y_1 \rangle \rangle  - \langle \langle x_2, y_2 \rangle \rangle.\end{equation}
A Lie subalgebra $\mathfrak{l}\subset\fd$ is said to be {\it Lagrangian} if $\mathfrak{l}$ is maximal isotropic with respect to the form (\ref{form}).
In other words, $\mathfrak{l}$ is Lagrangian if $\dim_\C \mathfrak{l} = \dim_\C\mathfrak{g}$ and $\langle x, y\rangle = 0$ for all $x, y \in \mathfrak{l}$.

A \e{Lagrangian splitting} of $\fd = \fg \oplus \fg$ is a vector space decomposition 
$\fd = \fl_1 + \fl_2$ 
in which both $\fl_1$ and $\fl_2$ are Lagrangian (note that it is not assumed that  $\fd$ is isomorphic to $\fl_1 \oplus \fl_2$ as Lie algebras). It will be helpful to observe that  the form (\ref{form}) yields an isomorphism $\fl_2 \srm{\isom} (\fl_1)^*$. It is clear that the diagonal Lie subalgebra 
$\mathfrak{g}_\Delta=\{(x,y)\in \fd \ |\ x=y\}$ is Lagrangian in $\fd$.  A {\it Belavin--Drinfeld splitting}, or just BD splitting, is a Lagrangian splitting $\fd = \mathfrak{l}_1 + \mathfrak{l}_2$
where $\mathfrak{l}_1=\mathfrak{g}_\Delta$.  In \cite[Example 4.4]{EL2} BD splittings are classified via \cite{BD}.  There is always at least one such splitting, namely the {\it standard Lagrangian splitting} $\mathfrak{l}_2\subset \mathfrak{b}\oplus\mathfrak{b}^{-}$ where $\mathfrak{b}$, $\mathfrak{b}^{-}$ are opposite Borel subalgebras of $\mathfrak{g}$ (see \cite{EL2} for details).

In \cite{EL1, EL2} it is shown that each Lagrangian splitting of $\fd$ endows $\overline{G}$ with a Poisson structure. Moreover, they show that each of these Poisson structures restricts to a Poisson structure on each $(G\times G)$--orbit, and 
hence to each boundary divisor in $\ol{G}$. We now review this construction.

For a complex manifold $M$, a bracket on the ring of holomorphic functions is a Poisson bracket if and only if the associated bivector $\Lambda$ 
satisfies $[\Lambda, \Lambda] = 0$, where $[\Lambda, \Lambda]\in \Lambda^3 (T^{1,0}M)$ is the Schouten bracket of $\Lambda$ with itself (see \cite[Theorem 1.8.5]{DuZu}). We will say that $\Lambda$ is a Poisson bivector when $[\Lambda, \Lambda] = 0$.
To simplify notation, given a holomorphic map $f\co M\to N$ of complex manifolds, we write $f_*$ to denote both the derivative $Df$ of $f$ and the maps on higher-order tensor fields induced by $Df$.

Let $\L_\fd\subset \mathrm{Gr}(n,\fd)$ be the space of Lagrangians in $\fd = \fg \oplus \fg$. Clearly $\L_\fd$ is a subvariety of the Grassmannian $\mathrm{Gr}(n,\fd)$.  Following the construction in \cite{EL1, EL2}, 
the Evens--Lu bivector $\Lambda$ on $\L_\fd$ is  defined by choosing a basis $\{x_i\}_i$ for $\mathfrak{l}_1$, and letting $\{y_i\}$ be the dual basis for $\mathfrak{l}_2 \isom \fl_1^*$ (that is, $\{y_i\}$ is the basis satisfying   $\la x_i, \xi_j\ra = \delta_{ij}$). Now define
$$r = \dfrac{1}{2} \sum_i x_i \wedge y_i \in \Lambda^2 (\fg \oplus \fg)$$
and
$$\Lambda_\fl = (\rho_\fl)_* (r) = \dfrac{1}{2} \sum_{i} (\rho_\fl)_* (x_i) \wedge (\rho_\fl)_* (y_i) \in \Lambda^2 (T_\fl \L_\fd),$$ where $\rho_\fl$ is defined below.
We note that $\Lambda_\fl$ is independent of the choice of basis $\{x_i\}_i$: for instance, using the form (\ref{form}), we may view $r = \frac{1}{2} \sum_i x_i \wedge y_i$ as an element of $(\Lambda^2 \fd)^*$, and evaluating $r$ on an element $(v_1, f_1) \wedge (v_2, f_2)\in \Lambda^2 (\fl_1 \oplus \fl_1^*) \isom \Lambda^2 (\fd)$ gives $f_1 (v_2) - f_2 (v_2)$, as can be checked on the basis for $\Lambda^2 (\fl_1 \oplus \fl_1^*)$ constructed from $\{x_i\}_i$.

As discussed in \cite[Examples 4.3 and 4.4]{EL2}, this bivector induces a Poisson structure on $\L_\fd$ and on each $G\cross G$ orbit in $\L_\fd$, as well as on the closure of each orbit. In particular, $\Lambda$ induces a Poisson structure on $\ol{G}$, which is the closure of the orbit $(G\cross G)\cdot \fg_\Delta$ of the diagonal $\fg_\Delta \in \L_\fd$.

Our next goal is to understand how the Evens--Lu Poisson structure interacts with the 
action of $G\cross G$ on $\ol{G}$, which is induced by the inclusion 
$$G\cross G \hookrightarrow \mathrm{Aut}(\fg) \cross  \mathrm{Aut}(\fg) \subset  \mathrm{Aut}(\fg \oplus \fg).$$ 
We recall some terminology regarding Poisson Lie groups and Poisson actions.
Let $M_1$ and $M_2$ be two Poisson varieties. A morphism $M_1\to M_2$ is called a {\it Poisson morphism} (or ichthyomorphism) if the dual morphism $\mathcal{O}(M_2) \to \mathcal{O}(M_1)$ is a morphism of Poisson sheaves.  A Poisson-algebraic group is an algebraic group $G$, equipped with a Poisson structure for which the group multiplication $G \times G \to G$ is a Poisson map. The action of a Poisson-algebraic group $G$ on a Poisson variety $M$ is a {\it Poisson action} when the action map $\alpha : G\times M \to M$ is a Poisson map, where $G\times  M$ has the product Poisson structure (defined by the sum of bivectors).

We introduce some notation that will be needed in the next lemma. Consider the (left) action of $G\cross G$ on  $\L_\fd$.  For each $\mathfrak{l}\in \L_\fd$, let
$$\rho_{\mathfrak{l}} \co G\cross G \to \L_\fd$$
be the map
$$\rho_{\mathfrak{l}} (g, h) = (g,h)\cdot \mathfrak{l}.$$
For each $(g,h)\in G\cross G$, let
$$\rho_{(g,h)} \co \L_\fd \maps \L_\fd$$
be the map
$$\rho_{(g, h)} = (g,h)\cdot \fl,$$
and let
$$\mu^R_{(g,h)} \co G\cross G \maps G\cross G \textrm{ and } \mu^L_{(g,h)} \co G\cross G \maps G\cross G$$
be the maps given by right- and left-multiplication by $(g,h)$ (respectively).

Define the BD--bivector on $G\cross G$ associated to the data $\{x_i\}_i$, $\{y_i\}_i$ by
\begin{equation}\label{Pi} \Pi_{(h,k)} = \dfrac{1}{2} \sum_i \left[ (\mu^R_{(h,k)})_* (x_i  \wedge y_i) -  (\mu^L_{(h,k)})_* (x_i \wedge y_i) \right].
\end{equation}

Similar to the discussion of the Evens--Lu bivector, the BD--bivector is independent of the choice of basis and so only depends on the BD splitting.

We will need the following standard fact.

\begin{lemma} 
The bivector $\Pi$ is Poisson, and induces a Poisson--Lie group structure on $G\cross G$ called the BD--Poisson structure.
\end{lemma}

This fact is discussed in various places in the literature. As discussed in~\cite[Section 2]{Lu-Mouquin}, the element $\dfrac{1}{2} \sum_i x_i \otimes y_i\in \fd\otimes \fd$ is a \e{quasitriangular} $r$--matrix, and
a quasitriangular $r$--matrix always induces a Poisson Lie group structure via the construction (\ref{Pi}), see \cite[pp. 46-47]{Kosmann-Schwarzbach}. Another discussion of this fact can be found in~\cite[Proposition 3.4.1]{KS}.

The next lemma is a version of Proposition 2.17 in~\cite{EL1}.

\begin{lemma}\label{545} For any BD--splitting of  $\fg \oplus \fg$, the action of $G\cross G$ on $\L_\fd$ is a Poisson action, where $G\cross G$ has the BD--Poisson structure and $\L_\fd$ has the Evens--Lu Poisson structure.
\end{lemma}
\begin{proof} When written in terms of bivectors, the condition for the action to be Poisson becomes
\begin{equation}\label{DZ-act}
\Lambda_{(g,h) \cdot \fl} = (\rho_{(g,h)})_* (\Lambda_\fl) + (\rho_\fl)_* (\Pi_{(g,h)})
\end{equation}
(see, for instance, \cite[5.4.5]{DuZu}).

 To prove (\ref{DZ-act}),
 observe that
$$\rho_{(g,h)\cdot \fl} = \rho_\fl \circ \mu^R_{(g,h)}$$
 and
 $$\rho_{(g,h)} \circ \rho_\fl = \rho_\fl \circ \mu^L_{(g,h)}.$$
  We now have
   $$\Lambda_{(g,h) \cdot \fl} - (\rho_{(g,h)})_* (\Lambda_\fl) = (\rho_{(g,h) \cdot \fl})_* (r) - (\rho_{(g,h)})_* \left( (\rho_\fl)_* (r) \right) $$
$$   = (\rho_\fl)_* ( (\mu^R_{(g,h)})_* (r)) - (\rho_\fl)_* ( (\mu^L_{(g,h)})_* (r))   
   =(\rho_\fl)_* (\Pi_{(g,h)}).$$
\end{proof}

We now turn to the conjugation action of $G$ on $\ol{G}$, which extends the conjugation action of $G$ on itself. Recall that $\ol{G}$ is the closure of $(G\cross G)\cdot \fg_\Delta$ inside $\L_\fd$. The map
$$G\maps  (G\cross G)\cdot \fg_\Delta$$
defined by $g\mapsto (g,e)\cdot  \fg_\Delta$ is a diffeomorphism.
If we give $G$ the $G\cross G$ action
$$(h,k)\cdot g = hgk^{-1},$$ 
then
this diffeomorphism is $G\cross G$--equivariant. In particular, the action of the subgroup $G_\Delta \subset G\cross G$ on $\ol{G}$ corresponds, under this diffeomorphism, to the (left) conjugation action of $G$ on itself; $h\cdot g = hgh^{-1}$.

We wish to study the conjugation action of $G$ on $\ol{G}^{\,n}$ and its interaction with the Evens--Lu Poisson structure. However, a subtlety arises: if we equip $\ol{G}^{\,n}$  and  $(G\cross G)^n$  with the direct product Poisson structures arising from a BD--splitting of $\fd$, then the action of $(G\cross G)^n$ on $\ol{G}^{\,n}$ is Poisson, but this does not imply that the action of the 
diagonal subgroup $\{(g,g,\ldots, g)\} \subset (G\cross G)^n$ is Poisson, as this diagonal subgroup need not be a Poisson Lie subgroup. 

Recent work of Lu and Mouquin~\cite{Lu-Mouquin} provides a way to avoid this problem by using the \e{mixed product} Poisson structure on $\ol{G}^{\,n}$.
We briefly explain the setup, specialized to our situation. Details may be found in~\cite[Section 6]{Lu-Mouquin}.
 Let $G$ be as above, and equip $D=G\cross G$ with the above Poisson structure. Given a Poisson $D$--space $(Z, \pi_Z)$, let $\lambda \co \fd\to \mathcal{V}^1 (Z)$ be the map induced by the action, sending $x\in \fd$ to the vector field $\dfrac{d}{dt}|_{t=0} \exp(tx)y$; see~\cite[Section 1.4]{Lu-Mouquin}. 

Lu and Mouquin define the mixed product Poisson bivector on $Z^n$ by the formula
$$\pi_{Z^n} = (\pi_Z, \ldots, \pi_Z) + \sum_{1\leq j < k \leq n} \sum_i  (i_j)_*\lambda (y_i) \wedge (i_k)_* (\lambda (x_i))$$
where $r = \sum_i x_i\otimes y_i$ is the $r$--matrix defining the Poisson structure on $D$ and $i_l\co Z\to Z^n$ is the inclusion into the $l$th factor of the product. By~\cite[Theorem 6.8 and Lemma 2.13]{Lu-Mouquin}, the diagonal action of $D$ on $(Z^n, \pi_{Z^n})$ is a Poisson action. In particular, letting $Z = \ol{G}$ with $\pi_Z = \Lambda$ (the Evens--Lu Poisson bracket), we find that the diagonal action of $D$ on $(\ol{G}^{\,n}, \Lambda_n)$ is Poisson, where $\Lambda_n$ is the mixed product Poisson structure on $\ol{G}^{\,n}$ associated to $\Lambda$. The diagonal subgroup $G_\Delta \subset D$ (corresponding to the Lagrangian subalgebra $\fg_\Delta \subset \fd = \fg \oplus \fg$) is a Poisson Lie subgroup of $D$, as explained (for instance) in~\cite[Appendix]{Evens-Lu-GR}. Returning to the general setting above, this implies that the diagonal action of $D$ on $Z^n$ restricts to a Poisson action of $G = G_\Delta$ on 
$(Z^n, \pi_{Z^n})$; note that this is precisely the action of $G$ given by the diagonal embedding of $G$ into $D^n = G^{2n}$. In our case, these facts lead to the following result.

\begin{proposition}\label{adinvar} Let $G$ be a semisimple group of adjoint type, and fix a BD-splitting of $\fg\oplus \fg$, with associated quasitriangular $r$--matrix $r\in \Lambda^2 (\fg\oplus \fg)$. 
Equip $G$ with the Poisson structure induced by $r$, and equip $\ol{G}^{\,n}$ with the mixed product Poisson structure associated to the Evens--Lu Poisson structure on $\ol{G}$.
Then the diagonal action of $G$ on $\ol{G}^{\,n}$ is Poisson, and restricts to the diagonal conjugation action of $G$ on $G^n \subset \ol{G}^{\,n}$.
\end{proposition}

It is well-known that if $X$ is a Poisson manifold and a Lie group $G$ acts on $X$ through Poisson maps, then the $G$--invariant functions on $X$ form a Poisson algebra (see for instance~\cite[p. 24]{DuZu}). The following proposition is a version of this statement. 

\begin{proposition}\label{lemmapoisson2}
Let $X$ be a quasi-projective Poisson variety and let $G$ be a reductive algebraic group that is a Poisson Lie group. If $G$ acts on $X$ and the action map $X\cross G \to X$ is Poisson, then with respect to any $G$-linearized ample line bundle $L$, the GIT quotient $X\quot_L G$ is a Poisson variety and the quotient map $X\to X\quot_L G$ is a Poisson map.
\end{proposition}

\begin{proof} 
This is a consequence of Property $(1)$ in \cite[Lemma 5.4.5]{DuZu} which characterizes Poisson actions in terms of bivectors.  The explicit statement, in the affine case, is given in \cite[Proposition 5.33]{LPV}.  In the quasi-projective case,  $X\quot_L G$ is built from open affine subvarieties (see \cite[Section 8.2]{Do}), so one can apply the affine case locally. A detailed discussion is provided in the Appendix.
\end{proof}

\begin{theorem}$\label{poisson-thm}$
There exists a Poisson structure on the wonderful compactification of a free group character variety over $\C$, and also on its boundary divisors.
\end{theorem}

\begin{proof}

The Poisson structure on $\overline{\fX_{F_r}(G)}$ follows directly from Proposition~\ref{adinvar} and Proposition~\ref{lemmapoisson2}.

Since each boundary divisor of $\overline{G}^{\,r}$ is a union of products of orbits where each admits a Poisson structure (restricted from that on $\overline{G}$), the same argument as above shows that the Poisson structures on the boundary divisors of $\overline{G}^{\,r}$ descend to the boundary divisors of $\overline{\fX_{F_r}(G)}$.
\end{proof}

Since the boundary divisors of $\overline{\fX_{F_r}(G)}$ are unions of parabolic free group character varieties, we immediately conclude:

\begin{corollary}\label{poissoncor}
There exists a Poisson structure on those parabolic character varieties of free groups that lie inside the boundary divisors of $\overline{\fX_{F_r}(G)}$.
\end{corollary}

We call the Poisson structures shown to exist above the {\it wonderful Poisson structures.}

In \cite{G7,G8} Goldman showed there is a Poisson structure on $\hm(\pi_1(\Sigma_{g,n}),G)\quot G$ where $\Sigma_{g,n}$ is an orientable surface of genus $g$ with $n$ disjoint boundary components (see also \cite{La2}). Moreover, the Casimirs (those functions that Poisson commute) are exactly the invariant functions restricted to the boundary components. 

\begin{question}
How does Goldman's Poisson structure on $G^r\quot G$ relate to the wonderful Poisson structures on $G^r\quot_{\mathcal{L}} G$, and $\overline{G}^{\,r}\quot G$? 
\end{question}

\begin{remark}\label{Spec-rmk}
Given an affine Poisson variety $V$, the Poisson bracket $\{\ ,\ \}_V$ is determined by its action on the coordinate ring $\C[V]$ by the Stone-Weierstrass theorem. Suppose $V$ has Casimirs $\{c_1,
\ldots ,c_m\}$. Then the algebra $A:=\C[V]/(c_1-\lambda_1,\ldots,c_k-\lambda_k)$, where
$\lambda_1,\ldots,\lambda_k\in \C$, is a Poisson algebra with bracket defined by $\{f+I ,g+I \}=\{f,g\}_V+I$ where $I$ is the ideal $(c_1-\lambda_1,\ldots,c_k-\lambda_k)$. Therefore, the variety $\mathrm{Spec}(A)$ is an affine Poisson variety.
\end{remark}

Now applying Remark \ref{Spec-rmk} to the setting of parabolic character varieties of free groups we see that whenever the parabolic data $\{G/H_i\}$ are isomorphic to $G$-conjugation orbits $($equivalently $H_i$'s are isomorphic to conjugation stabilizers$)$, then the Goldman Poisson bracket on $G^r\quot G$ with some set of its Casimirs fixed $($fixing some set of the boundaries up to conjugation is equivalent to fixing some set of the Casimirs$)$ determines a Poisson structure on the parabolic character variety of a free group resulting from fixing some $($but not all$)$ the boundaries to conjugation orbits.
Therefore, we have a Goldman-type Poisson structure on certain parabolic character varieties of free groups. 

\begin{question} How does this Goldman-type Poisson structure compare to the wonderful Poisson structures from Corollary \ref{poissoncor}?
\end{question}

\appendix

\section{Poisson Structures on GIT Quotients \texorpdfstring{\\}{} by Arlo Caine and Sam Evens}
\label{sec_intro}

We explain how to put a Poisson structure on a quotient of a linearized irreducible Poisson algebraic variety by the action of a reductive Poisson algebraic group $G$.  First we discuss the affine setting, and then we apply the affine
case  to
the general situation.

\subsection{Quotient of an affine variety}
\label{sec_construction}

We explain how to put a Poisson structure on the quotient of an affine variety. These results are in \cite{Yang} and in \cite{LPV}.

As above, let $(G, \pi_G)$ be a reductive 
Poisson linear algebraic group. Denote the Poisson Lie algebra structure on the coordinate ring $\Bbbk[G]$ by $\{ \phi_1, \phi_2 \}_G$ for $\phi_1, \phi_2 \in \Bbbk[G]$. 

Let $(X, \{ , \}_X )$ be a Poisson algebraic variety, i.e., $\{ , \}$ makes the sheaf of regular functions ${\cO}_X$ into a Poisson algebra. Assume that $X$ is a $G$-variety with action map $a:G \times X \to X$,
and denote by $p:G\times X \to X$ the projection $p(g,x)=x$. The sheaf of functions ${\cO}_{G\times X} = {\cO}_G \otimes_{\Bbbk} {\cO}_X$ then acquires the structure of a Poisson Lie algebra, which is uniquely determined by the property (see \cite{KS}, Proposition 1.2.10, p. 9):
$$
(*) \ \ \{ \phi_1 \otimes f_1, \phi_2 \otimes f_2 \} = \{ \phi_1, \phi_2 \}_G \otimes f_1 f_2
+ \phi_1 \phi_2 \otimes \{ f_1, f_2 \}_X, \ \phi_i \in {\cO}_G, f_i \in {\cO}_X.
$$

Suppose for the remainder of this subsection that $X$ is affine, so we may work with regular functions $\Bbbk [G \times X] \cong \Bbbk[G] \otimes_{\Bbbk} \Bbbk[X]$. Note that $p^*(f)=1\otimes f$ using this
identification. By formula $(*)$, it follows that $p:G\times X \to X$ is Poisson.

\begin{remark}
\label{rem_invariance}
Let $f\in \Bbbk[X]$, and $\Bbbk[X]^G$ the ring of $G$-invariant functions on $X$.  Then $f\in \Bbbk[X]^G$ if and only if $p^*(f)=a^*(f)$.
\end{remark}

\begin{lemma}
\label{lem_invariantpalg}
Let $X$ be an affine Poisson $G$-variety. If $a:G\times X \to X$ is a Poisson morphism, then $\Bbbk [X]^G$ is a Poisson subalgebra of $\Bbbk[X]$.
\end{lemma}

\begin{proof}
Since $a$ is a Poisson morphism, $a^*\{ f_1, f_2 \}_X = \{ a^* f_1, a^* f_2 \}_{G\times X}$ for $f_1, f_2 \in \Bbbk [X]$. Suppose $f_1, f_2 \in \Bbbk [X]^G$. Using Remark \ref{rem_invariance}, $a^*(f_i)=p^*(f_i)=1\otimes f_i$. It follows that

$$
a^*\{ f_1, f_2 \}_X = \{ 1\otimes f_1, 1 \otimes f_2 \}_{G\times X} = 1\otimes \{ f_1, f_2 \}_X
= p^* \{ f_1, f_2 \}_X.
$$

Again by Remark \ref{rem_invariance}, $\{ f_1, f_2 \}_X \in \Bbbk[X]^G$.
\end{proof}

Now assume that $G$ is reductive. Then $\Bbbk[X]^G$ is a finitely generated $\Bbbk$-algebra, and by definition the geometric invariant theory quotient $X\quot G = \mathrm{Spec} (\Bbbk[X]^G)$, or in other words, $X\quot G$ is the affine variety with ring of regular functions $\Bbbk [X\quot G]=\Bbbk [X]^G$. There is a quotient morphism $q:X \to X\quot G$ with the property that $q^*:\Bbbk[X\quot G] \to \Bbbk[X]$ is the inclusion of invariant functions.

By  Lemma~\ref{lem_invariantpalg}, $\Bbbk[X\quot G]$ is a Poisson algebra, so $X\quot G$ is a Poisson algebraic variety. Since the inclusion $q^*:\Bbbk[X\quot G] \to \Bbbk[X]$ is Poisson, it follows
that $q:X \to X\quot G$ is Poisson.  Therefore we have proved the following proposition.

\begin{proposition}
\label{prop_affinequotient}
If $(G,\pi_G)$ is a reductive Poisson linear algebraic group and $(X, \{ , \}_X )$ is an affine Poisson algebraic variety, and the action map $G\times X \to X$ is Poisson, then $X\quot G$ is a Poisson algebraic variety, and the morphism $q:X \to X\quot G$ is Poisson.
\end{proposition}

\subsection{Quotient of a \texorpdfstring{$G$}{G}-linearized variety}

\label{sec_projectivequotient}

In this section, we explain how to put a Poisson structure on a GIT quotient of a linearized irreducible $G$-variety $X$.  Recall the notions of $G$-linearized
line bundle $L$ on $X$ and semistable locus from Section \ref{wonchar}.  The semistable
locus $X^{ss}_L = \cup_{s_i} X_{s_i}$ is a finite union of open affine $G$-stable
subsets $U_{s_i}$ of $X$, where $U_{s_i}$ is the non-vanishing locus of the
section $s_i$ of a power of $L$.   Let $Y_{s_i}:= U_{s_i}\quot G$ be the quotient
of the affine $G$-variety $U_{s_i}$.   Then the quotient $X\quot_L G$ has an open
affine cover $X\quot_L G = \cup Y_{s_i}$ (see \cite[Theorem 8.1]{Do}).

We remark that if we are given a Poisson structure on a 
variety $Z$, there is an induced Poisson structure on any open set
$U$ of $Z$.  Indeed, we may assume that $Z$ is affine 
and  $U$ is covered by affine open sets $Z_f :=
\{ x\in Z : f(x) \not = 0 \}$. The Poisson structure on $Z$
induces a Poisson Lie algebra structure on $\Bbbk[Z]$, and
we can define a Poisson Lie algebra structure on $\Bbbk[Z_f]$
by the formula given in the proof of Lemma 1.3 in \cite{kaledin}.
These Poisson structures glue together on $Z_f \cap Z_g = Z_{fg}$
and hence define a Poisson structure on the open set $U$.

\begin{proposition}
\label{prop_projgitpoisson}
Let $X$ be a an irreducible Poisson $G$-variety with $G$-linearization $L$, 
where $(G,\pi_G)$ is a reductive Poisson algebraic
group and the action morphism $a:G\times X \to X$ is
a Poisson morphism. Then $X^{ss}_L$ and $X\quot_L G$ are Poisson and $q:X^{ss}_L \to X\quot_L G$
is a Poisson morphism.
\end{proposition}

\begin{proof}
There is a finite set of $G$-invariant sections $s_i$ with the property that
the non-vanishing locus 
$X_{s_i}$ of $s_i$, is open, affine and $G$-stable.   Hence by
Proposition \ref{prop_affinequotient}, $Y_i := U_{s_i}\quot G$ is
an affine Poisson variety.  Thus, we have a Poisson structure $\pi_i$ on each open set
$Y_i$ in the
open cover $X\quot_L G = \cup  Y_i$.
The functions on the intersections $Y_i \cap Y_j$ form
a subring in the fraction field $\Bbbk (Y_i)$, and the
above formula from \cite{kaledin} implies that $\pi_i$ and
$\pi_j$ coincide on the sheaf of functions on $Y_i \cap Y_j$
and thus glue to give 
an induced Poisson structure on $X\quot_L G$. Since the morphism
$q:X \to Y$ is Poisson on the affine cover $U_i \to Y_i$ for
each of our invariant sections $s_i$, it follows that $q$ is a
Poisson morphism.  
\end{proof}

\def\cdprime{$''$} \def\cdprime{$''$} \def\cprime{$'$} \def\cprime{$'$}
  \def\cprime{$'$} \def\cprime{$'$}

\end{document}